\documentclass[final]{siamltex}

\usepackage[dvips]{epsfig}

\newcommand{\bh}{{\bf h}}

\newcommand{\bxi}{{\mbox{\boldmath $\xi$}}}
\newcommand{\bzet}{{\mbox{\boldmath $\zeta$}}}

\newcommand{\bep}{{\mbox{\boldmath $\epsilon$}}}

\newcommand{\bud}{{\underline {\bf d}}}

\newcommand{\bd}{{\bf d}}

\newcommand{\bU}{{\bf U}}

\newcommand{\bE}{{\bf E}}

\newcommand{\bX}{{\bf X}}

\newcommand{\E}{I \! \! E}
\newcommand{\R}{I \! \! R}

\newcommand{\C}{I \! \! \! \! {C}}

\newcommand{\us}{{\underline s}}

\newcommand{\uc}{{\underline c}}

\newcommand{\uq}{{\underline q}}

\newcommand{\um}{{\underline m}}

\newcommand{\umu}{{\underline \mu}}

\newcommand{\uy}{{\underline y}}
\newcommand{\uw}{{\underline w}}

\newcommand{\ud}{{\underline d}}
\newcommand{\ue}{{\underline e}}

\newcommand{\tz}{{\tilde{z}}}
\newcommand{\p}{{p^*}}

\newcommand{\uxi}{{\underline{\xi}}}
\newcommand{\uzet}{{\underline{\zeta}}}
\newcommand{\tzet}{{\tilde{\zeta}}}
\newcommand{\ugam}{{\underline{\gamma}}}

\newcommand{\bb}{\begin{eqnarray}}
\newcommand{\be}{\end{eqnarray}}

\title{Kernel density estimation  via diffusion and the complex exponentials approximation problem}

\author{Piero Barone\thanks{Istituto per le Applicazioni del Calcolo ''M. Picone'',C.N.R., Via dei Taurini 19, 00185 Rome, Italy,({\tt p.barone@iac.cnr.it, piero.barone@gmail.com})
}}

\begin{document}

\maketitle

\begin{abstract}
A  kernel method is proposed to estimate the condensed density of the generalized eigenvalues of  pencils of Hankel matrices whose elements have
a joint noncentral Gaussian distribution with nonidentical covariance. These pencils arise when the complex exponentials approximation problem is considered in Gaussian noise.
Several moments problems can be formulated in this framework and the estimation of the condensed density above is the main critical step for their solution. It is shown that the condensed density satisfies approximately a
diffusion equation, which allows to estimate an optimal bandwidth. It is proved by simulation that good results can be obtained even when the signal-to-noise ratio is so small that other methods fail.

\end{abstract}



\pagestyle{myheadings}
\thispagestyle{plain}

\section*{Introduction}
Many difficult moments problems such as the trigonometric,  the complex, the
Hausdorff ones can be formulated as the complex exponentials
approximation problem (CEAP), which can be stated as follows, denoting random variables by bold characters:
given  a uniformly sampled signal made up of a linear combination
of complex exponentials
\begin{eqnarray}s_k=\sum_{j=1}^\p c_j\xi_j^k.\label{modale}\end{eqnarray}
where $c_j,\xi_j\in\C,$ let us assume to know an even number $n=
2p,\;p\ge \p$ of noisy samples
$$\bd_k=s_k+\bep_k,\quad k=0,1,2,\dots,n-1$$ where $\bep_k$ is a complex Gaussian,
zero mean, white noise, with finite known variance $\sigma^2$.  The CEAP problem consists in  estimating  $\p,c_j,\xi_j,\;j=1,\dots,\p$. This is a well
known ill-posed inverse problem often addressed in the literature, see e.g.
\cite{donoho,gmv,hen2,osb,bardsp,barja2,barja,siam03,bama2,bama1,baram,maba2,maba1,vpb}.

We notice that, in the noiseless
case and when $p=\p$, the parameters $\xi_j$ are the generalized
eigenvalues of the pencil $(U_1,U_0)$ where $U_1$ and $U_0$ are Hankel matrices defined as
\bb
U_0=\left[\begin{array}{llll}
s_0 & s_{1} &\dots &s_{p-1} \\
s_{1} & s_{2} &\dots &s_{p} \\
. & . &\dots &. \\
s_{p-1} & s_{p} &\dots &s_{n-2}
  \end{array}\right],\;\;
U_1=\left[\begin{array}{llll}
s_1 & s_{2} &\dots &s_{p} \\
s_{2} & s_{3} &\dots &s_{p+1} \\
. & . &\dots &. \\
s_{p} & s_{p+1} &\dots &s_{n-1}
  \end{array}\right]\label{pencil}\be
If we define  $\bU_1$ and $\bU_0$ as $U_0$ and $U_1$ but starting from $\bd_k,\;k=0,\dots,n-1$
 it is evident that
  the generalized eigenvalues of the pencil $(\bU_1,\bU_0)$
provide information about the location in the complex plane of the
generalized eigenvalues $\xi_j,\;j=1\dots,p$ whose estimation is the most
difficult part of CEAP.

To make precise this statement, let us consider a $p\times p$ random matrix
$\bU$  and denote by  $\{\bxi_j,\;j=1,\dots,p\}$ its eigenvalues  which form a set of exchangeable random variables. Their marginal density $h(z),\;z\in \C$, also called condensed density
\cite{ham}, is the expected value of the
(random) normalized counting measure on the zeros of $\bU$ i.e.
$$h(z)= \frac{1}{p}E\left[\sum_{j=1}^{p}\delta(z-\bxi_j)\right]$$ or,
equivalently, for all Borel sets $A\subset\C$ $$\int_A
h(z)dz=\frac{1}{p}\sum_{j=1}^p Prob(\bxi_j\in A).$$
If a pencil $\bU=(\bU_1,\bU_0)$  of  random matrices is considered, the condensed density of its generalized eigenvalues can be computed by the formula above. Its relative maxima provide information about the location in the complex plane of the
generalized eigenvalues $\xi_j,\;j=1\dots,p$.
In \cite{barja2,bardsp,distrf} the use of this function for solving several moments problems was illustrated. However
when the  signal-to-noise ratio (SNR) measured e.g. by $SNR=\min_j\frac{|c_j|^2}{\sigma^2}$ is low the computation of the condensed density  is very
difficult  even if we assume to have replicated observations as we do in the following. The main problem is that many relative maxima related to noise are present in the condensed density which are slightly less intense than those related to the true signals. As we are interested only on signal related relative maxima of the condensed density we look for an estimation method with noise filtering abilities.

The aim of the paper is to show that  it is possible to devise a kernel density estimation method which does have such properties. In \cite{basisc} the same project was developed for the real exponentials approximation problem which is better known as the exponential analysis problem. The idea was to use a method proposed by Botev et al. \cite{botev}, based on a class of kernels which satisfy a diffusion PDE, which allows to automatically estimate the optimal bandwidth. It was shown that the generalized eigenvalues of the pencil $\bU$ can be approximated by the ratio of Gaussian variables whose density satisfies a specific diffusion PDE. In the limit for $SNR\downarrow\infty$ this PDE belongs to the class considered in \cite{botev} and their method can be used.

In the following a different approach is proposed. An explicit expression of the condensed density is derived and approximated by Laplace method. Then  an anisotropic diffusion PDE is derived which admits the approximated condensed density as a solution in the limit for $\sigma\downarrow 0$. Finally an optimal bandwidth is derived on the same lines of Botev et al. results. We can then propose  an estimator which has better filtering abilities than the standard one based on Gaussian kernel.

The paper is organized as follows. In Section 1 the condensed density of the generalized eigenvalues of $\bU$ is derived. In Section 2 its Laplace approximation is computed.  In
Section 3 the anisotropic diffusion PDE is derived. In Section 4 the kernel estimator and the optimal bandwidth are computed. In Section 5 the proposed algorithm is illustrated. Finally in Section 6 two numerical examples are discussed.
The Appendix contains  the proof of most  theorems and lemmas.

\section{The  condensed density of the generalized eigenvalues of $\bU$ }

From \cite[eq.(7)]{barja2} the condensed density of the generalized eigenvalues of $\bU$  is
$$
h_n(z,\sigma)=\frac{2}{n}E\left[\sum_{j=1}^{n/2}\delta(z-\bxi_j)\right]=
\frac{2}{n}\sum_{j=1}^{n/2}E\left[\delta(z-\bxi_j)\right]=\frac{2}{n}\sum_{j=1}^{n/2}h_n^{(j)}(z,\sigma)$$
where $$h_n^{(j)}(z,\sigma)=\frac{1}{(\pi\sigma^2)^{n}}
\int_{\C^{n/2-1}}\int_{\C^{n/2}}
J_C^*(\uzet^{(j)},z,\ugam)e^{-\frac{1}{\sigma^2}
\sum_{k=0}^{n-1}\left|\sum_{h\ne j}^{1,n/2} \gamma_h\zeta_h^{k}+\gamma_j
z^{k}-s_k\right|^2} d\uzet^{(j)}d\ugam $$
and
 $\uzet^{(j)}=\{\zeta_h,h\ne j\}$ and $$J_C^*(\uzet^{(j)},z,\ugam)=
\left\{\begin{array}{cc} \gamma & \mbox{ if } n=2\\
(-1)^{n/2}\prod_{ j=1}^{1,n/2}\gamma_j\prod_{r<h;r,h\ne
j}(\zeta_r-\zeta_h)^4\prod_{r\ne j}(\zeta_r-z)^4 & \mbox{ if } n\ge
4\end{array}\right.$$
In the following  Lemmas we give simpler forms and properties of the condensed density.
\begin{lemma}
\bb
h_n^{(j)}(z,\sigma)= \frac{1}{(\pi\sigma^2)^{n}}\int_{\C^{n/2-1}}\int_{\C^{n/2}}
J_C^*e^{-\frac{1}{\sigma^2}[(\ugam-\umu_j)^HQ_j(\ugam-\umu_j)+\nu_j]
} d\uzet^{(j)}d\ugam,  \nonumber\be
$$J_C^*=J_C^*(\uzet^{(j)},z,\ugam),\;\;\umu_j=\umu_j(\uzet^{(j)},z),\;\;Q_j=Q_j(\uzet^{(j)},z),\;\;\nu_j=\nu_j(\uzet^{(j)},z),$$
and
$$x_{hk}^{(j)}=\left\{\begin{array}{ll}
\zeta_h^{k-1}, & h\ne j\\
z^{k-1}, & h=j
\end{array}\right. $$
$$X_j(h,k)=\overline{x}_{hk}^{(j)},\;\;q_j=X_j\us,\;\;X_j\in\C^{n/2\times n},\;\;
Q_j=X_jX^H_j\in\C^{n/2\times n/2} ,$$
$$\umu_j=Q_j^{-1}\uq_j,\;\;\nu_j=\sum_{k=0}^{n-1}|s_k|^2-\umu_j^H Q_j \umu_j=\us^H(I_n-X_j^H(X_jX_j^H)^{-1}
X_j)\us.$$ \label{lem1}
\end{lemma}

\begin{lemma}
If
$$X(h,k)=\overline{\xi}_h^{k-1},\;\;\us=X^H\uc,\;\;\nu(\uxi)=\us^H(I_n-X^H(XX^H)^{-1}X)\us$$
then
$\nu(\uxi)=0.$ \label{cor1}
\end{lemma}
  \begin{proof} $\nu(\uxi)=\us^H(I_n-X^H(XX^H)^{-1}X)\us=\uc^HX(I_n-X^H(XX^H)^{-1}X))X^H\uc=0.$
\end{proof}

\begin{lemma}
  When $n=2$ the condensed density is given by
  $$h_2^{(1)}(z,\sigma)=e^{-\rho\frac{|z-\xi_1|^2}{1+|z|^2}}\left(\frac{\rho|1+\overline{z}\xi_1|^2}{\pi(1+|z|^2)^3} + \frac{1}{\pi(1+|z|^2)^2}  \right)$$
   where  $\rho=\frac{|c_1|^2}{\sigma^2}$ denotes the SNR. \label{cor2}
\end{lemma}

\begin{lemma}
When $n>2$
\bb h_n^{(j)}(z,\sigma)= \int_{\R^{n-2}}f_j(\uzet^{(j)},z)\frac{\prod_{r<h;r,h\ne j}|\zeta_r-\zeta_h|^8\prod_{r\ne j}|\zeta_r-z|^8 }{(\pi\sigma^2)^{n/2}|\tilde{Q}_j(\uzet^{(j)},z)|^{\frac{1}{2}}} e^{-\frac{1}{\sigma^2}\nu_j(\uzet^{(j)},z)} d\Re{\uzet^{(j)}}d\Im{\uzet^{(j)}}
\nonumber\be
where
$$f_j(\uzet^{(j)},z,\sigma)=\E\left[\prod_{ i=1}^{1,n/2} \tilde{\ugam}^T A_i \tilde{\ugam}\right],\;\;A_i=I_2\otimes
\ue_i\ue_i^T$$
$\E$ denotes the expectation with respect to the Gaussian density $N\left(\tilde{\umu}_j,\Sigma_j\right)$ where $\Sigma_j=\frac{\sigma^2}{2}\tilde{Q}_j^{-1},$
$$\tilde{Q}_j=\left[\begin{array}{rr}
\Re{Q_j} &  -\Im{Q_j} \\
\Im{Q_j} & \Re{Q_j} \end{array}\right]$$
and $\tilde{\ugam},\;\tilde{\umu}_j$ are obtained by stacking the real and imaginary parts of $\ugam$ and $\umu_j$ respectively.\label{lem2}
\end{lemma}

\begin{lemma}
$$f_j(\uzet^{(j)},z,\sigma)=P_{n/2}(\sigma^{2},z)=\sum_{h=0}^{n/2}\beta_h(\uzet^{(j)},z)\frac{\sigma^{2h}}{2^h D_j^{n-h}}$$
where  $\beta_h(\uzet^{(j)},z)$ are positive polynomials and $D_j=\det(\tilde{Q}_j)$.
Moreover $$\beta_0(\uzet^{(j)},z)=\prod_{i=1}^{n/2}\tilde{\uq}_j^T\hat{Q}_jA_i\hat{Q}_j\tilde{\uq}_j$$
and $\beta_{n/2}(\uzet^{(j)},z)$ is the only coefficient that does not depend on $\us$. \label{lemf}
\end{lemma}

\begin{corollary}  When $n>2$
\bb h_n^{(j)}(\tz,\sigma)= \int_{\R^{n-2}}g_j(\tilde{\uzet}^{(j)},\tz,\sigma) e^{-\frac{1}{\sigma^2}\nu_j(\tilde{\uzet}^{(j)},\tz)} d\tilde{\uzet}^{(j)}\label{eqdc}
\be
where $\tz$ and $\tilde{\uzet}^{(j)}$ are the vectors obtained by stacking the real and imaginary parts of $z$ and $\uzet^{(j)}$ respectively,  and $$ g_j(\tilde{\uzet}^{(j)},\tz,\sigma)=
\frac{1}{\sigma^{n}\pi^{n/2}}\sum_{k=0}^{n/2}\frac{\sigma^{2k}}{2^k}\beta_k(\tilde{\uzet}^{(j)},\tz)\frac{\prod_{r<h;r,h\ne j}|\tzet_r-\tzet_h|^8\prod_{r\ne j}|\tzet_r-\tz|^8 }{D_j^{n-k+\frac{1}{2}}(\tilde{\uzet}^{(j)},\tz)}$$
\label{lem3}
\end{corollary}

\begin{corollary} When $n>2,\;\;\forall j=1,\dots,n/2$
\bb h_n^{(j)}(\tz,\infty)&=&h_n^{(1)}(\tz,\infty)= \nonumber \\ &&\frac{1}{(2\pi)^{n/2}}\int_{\R^{n-2}}\beta_{n/2}(\tilde{\uzet}^{(1)},\tz)\frac{\prod_{r<h;r,h\ne 1}|\tzet_r-\tzet_h|^8\prod_{r\ne 1}|\tzet_r-\tz|^8 }{D_1^{\frac{n+1}{2}}(\tilde{\uzet}^{(1)},\tz)}  d\tilde{\uzet}^{(1)}\nonumber\label{dcinf}.
\be
Moreover this is also the condensed density obtained when $\us=\underline{0}$ i.e. when $\uc=\underline{0}$ and it is circularly symmetric i.e. it depends only on $|\tz|^2$.
\label{cor4}
\end{corollary}

\section{The Laplace approximation of the condensed density}

The expression of the condensed density given in Corollary \ref{lem3} can be approximated in the limit for $\sigma\downarrow 0$ as follows
\begin{theorem}
In the limit for $\sigma\downarrow 0$, for $\tz\in {\cal N}_j$ where ${\cal N}_j$ is a neighbor of $\xi_j$ we have
\bb h_n^{(j)}(\tz,\sigma)\approx \hat{h}_n^{(j)}(\tz,\sigma)=\sigma^{-2}G_j(\hat{\uzet}^{(j)},\tz)e^{-\frac{1}{\sigma^2}\nu_j(\hat{\uzet}^{(j)},\tz)}\be
where
$$G_j(\hat{\uzet}^{(j)},\tz)=(2 \pi)^{\frac{n}{2}-1} |H_j(\hat{\uzet}_j,\tz)|^{-\frac{1}{2}}K_j(\hat{\uzet}_j,\tz)$$
$$K_j(\tilde{\uzet}^{(j)},\tz)=\frac{1}{\pi^{\frac{n}{2}}}\beta_0(\tilde{\uzet}^{(j)},\tz)\frac{\prod_{r<h;r,h\ne j}|\tzet_r-\tzet_h|^8\prod_{r\ne j}|\tzet_r-\tz|^8 }{D_j^{n+\frac{1}{2}}(\tilde{\uzet}^{(j)},\tz)}$$
$H_j$ is the Hessian of $\nu_j$
and $\hat{\uzet}_j$ is the unique minimum of $\nu_j(\tilde{\uzet}^{(j)},\xi_j)$ in a neighbor ${\cal D}_j$ of $\{\xi_h,\;h\ne j\}$ and $\beta_0(\tilde{\uzet}^{(j)},\tz)$ is given in Lemma \ref{lemf}.
\end{theorem}

\section{The diffusion equation}

In \cite[Th.2]{barja2} it was proved that $h_n(z,\sigma)$ converges weakly to the positive measure $\frac{2}{n}\sum_{h=1}^p\delta(z-\xi_h)$ when $\sigma\downarrow 0$. By the definition of  $h_n^{(j)}(z,\sigma)$ the same proof implies that $h_n^{(j)}(z,\sigma)$ converges weakly to $\delta(z-\xi_j)$. Moreover from Corollary \ref{dcinf} we know that  for
 $\sigma\downarrow \infty$, $h_n^{(j)}(z,\sigma)$ converges to a density independent of $\uxi$. Therefore we can guess that when $\sigma$ moves from $0$ to $\infty$, $h_n^{(j)}(z,\sigma)$ diffuses from an atomic measure centered in $\xi_j$  to a measure circularly symmetric w.r. to zero.
We then look for a diffusion equation which admits $\hat{h}_n^{(j)}(z,\sigma)$ as solution for $\sigma \downarrow 0$.

\noindent Let us define $z=x+i y$ and consider the anisotropic diffusion $h_t=L[h]$ where
\bb L[h](x,y,t)=\mbox{div}\left[{a_j(x,y)\nabla{\left(\frac{h(x,y,t)}{p(x,y)}\right)}}\right]\label{diff}\be
and $p(x,y)=h_n^{(j)}(x,y,\infty)$ is the stationary probability density, and $a_j(x,y)>0$ is the unknown diffusion coefficient.
Substituting $h(x,y,t)$ in the equation above with $\hat{h}_n^{(j)}(x,y,\sigma),\;\;\;t=\sigma^2$ where now $\sigma$ is considered as a variable (not a fixed known value) and dropping the indices $j,n$ and the variables $(x,y)$ we get
   \bb\frac{C e^{-\frac{\nu }{t}}}{p^3}\left( \frac{E_1}{t}+\frac{E_2}{t^2}+\frac{E_3}{t^3}\right)=0
   \be
where
\bb E_1=-a_y G_y p^2-a_x G_x
   p^2+a_y G p_y p+a_x G p_x p+2
   a p \left(G_y p_y+G_x p_x\right)-\nonumber\\
   a
   \left(G_{yy}+G_{xx}\right) p^2+a G
   \left(p_{yy}+p_{xx}\right) p-2 a G
   \left(p_y^2+p_x^2\right)\nonumber\be
\bb E_2=a_y G p^2 \nu_y+a_x
   G p^2 \nu_x+2 a G_y p^2 \nu_y+2 a
   G_x p^2 \nu_x-2 a G p_y p \nu_y-\nonumber\\
   2
   a G p_x p \nu_x+a G p^2 \left(\nu
   _{yy}+\nu _{xx}\right)-G p^3\nonumber\be
  $$E_3=-a G p^2 \nu_y^2-a G p^2 \nu_x^2+G
   p^3 \nu $$
   In the limit for $t \downarrow 0$ the dominant term in the left side of the equation is $\frac{C e^{-\frac{\nu }{t}}}{p^3} \frac{E_3}{t^3}$, therefore the equation is approximately satisfied when $E_3=0$ or, equivalently when
   $$a(x,y)=\frac{p(x,y) \nu (x,y)}{\nu _{x}(x,y)^2+\nu _{y}(x,y)^2}>0.$$

   In the following Lemma we prove that the Csisz\'ar  distance between $h(x,y,t)$ and $p(x,y)$ defined as
   $${\cal D} (h,p)=\int_{\R^2} p(x,y)\Psi\left(\frac{h(x,y,t)}{p(x,y)}\right)dxdy,\;\;\Psi\in {\cal C}^2:\R^+\rightarrow\R^+,\;\;\Psi''(\cdot)>0,\;\;\Psi'(1)=0 $$
   is a monotonic decreasing function of $t$, therefore $h(x,y,t)$ tends monotically to $p(x,y)$ when $t\downarrow\infty$.
   \begin{lemma}
   $$\frac{\partial {\cal D} (h,p)}{\partial t}=-\int_{\R^2}a(x,y)\Psi''\left(\frac{h(x,y,t)}{p(x,y)}\right)\left\|\nabla\left(\frac{h(x,y,t)}{p(x,y)}\right)\right\|_2^2dxdy<0.$$
   \end{lemma}

\section{The kernel estimator}
   \noindent Given a sample of size $R$ of the data $$\bud^{(r)}=[\bd^{(r)}_1,\dots,\bd^{(r)}_n],\:\:r=1,\dots,R$$ where $\E[\bd^{(r)}_k]=s_k$, we consider the kernel estimator of
$h^{(j)}_n(z,t)$ with bandwidth $t$ given by
$$\bh^{(j)}_n(x,y,t)=\frac{1}{R}\sum_{r=1}^R\Phi_j(x,y,\bzet_j(r);t)$$
where $\Phi_j(x,y,\bzet_j(r);t)$ is a solution of the diffusion equation (\ref{diff}) at time $t$ with initial condition $\delta(z-\bzet_j(r))$
and $\bzet_j(r),\;j=1,\dots,n/2$ are the generalized eigenvalues obtained from the data $\bud^{(r)}$.
Hence $\bh^{(j)}_n(x,y,t)$ is a solution of the diffusion equation (\ref{diff}) at time $t$ with initial condition
\bb\bE_j(x,y)=\frac{1}{R}\sum_{r=1}^R\delta(z-\bzet_j(r))\label{ic}\be
Therefore $\bE_j(x,y)$ is the empirical distribution of the generalized eigenvalue $\bzet_j$ and
$$\E[\bE_j(z)]=\frac{1}{R}\sum_{r=1}^R\int_{\C}\delta(z-\zeta) h^{(j)}_n(\zeta)d\zeta=h^{(j)}_n(z)$$
where $h^{(j)}_n(z)$ is the unknown true density corresponding to the known fixed value of $\sigma$.

\noindent In order to find the optimal bandwidth  we need the form of the kernel $\Phi_j(x,y,\bud^{(r)};t)$
for $t\downarrow 0$. By construction, dropping the index  $r$,  we have
$$\Phi_j(x,y,\zeta;t)=\hat{h}^{(j)}_n(x,y,t)=t^{-1}G_j(x,y,\zeta)e^{-\frac{1}{t}\nu_j(x,y,\zeta)},\;\;t\downarrow 0.$$
Associated to the anisotropic diffusion \ref{diff} there exists a Markov process $\bX_t$ whose transition probabilities are given, when $t\downarrow 0$, by $\Phi_j(z,\zeta;t)$. Moreover when the initial density of $\bX_0$ is $\delta(z-\zeta)$, the density of $\bX_t$ is (\cite[eq.(5.1)Ch.X.5]{fel2})
$$\Phi_j(z,\zeta;t)=\int_{\C }\Phi_j(z,u;t)\delta(u-\zeta)du$$
Therefore the kernel  $\Phi_j(z,\zeta;t),\;\;\forall \zeta$ and $t>0$, satisfies the forward equation
\bb\left\{\begin{array}{ll}
\frac{\partial \Phi_j}{\partial t}(z,\zeta;t)-L[\Phi_j(z,\zeta;t)]=0\label{forw}\\
\Phi_j(z,\zeta;0)=\delta(z-\zeta)
\end{array}\right. .\be
Moreover the conditional expectation of $\delta(u-\bX_t)$ on the hypothesis that $\bX_0=z$ is (\cite[eq.(4.5)Ch.X.4]{fel2})
$$\Phi_j(z,\zeta;t)=\int_{\C }\Phi_j(u,\zeta;t)\delta(u-z)du$$
and therefore the kernel  $\Phi_j(z,\zeta;t),\;\;\forall z$ and $t>0$, satisfies the backward equation
\bb\left\{\begin{array}{ll}
\frac{\partial \Phi_j}{\partial t}(z,\zeta;t)-L^*[\Phi_j(z,\zeta;t)]=0\label{backw}\\
\Phi_j(z,\zeta;0)=\delta(z-\zeta)
\end{array}\right.\be
where
$$L[\Phi_j(z,\zeta;t)]=\mbox{div}_z\left[a_j(z)\nabla_z\left(\frac{\Phi_j(z,\zeta;t)}{p_j(z)}\right)\right]$$ i.e.  $$L[\Phi_j(z,\zeta;t)]=a_j(z)\Delta_z\left[\frac{\Phi_j(z,\zeta;t)}{p_j(z)}\right]+(\nabla_z a_j(z))^T\nabla_z \left[ \frac{\Phi_j(z,\zeta;t)}{p_j(z)}\right]$$ and
$$L^*[\Phi_j(z,\zeta;t)]=\Delta_\zeta\left[a_j(\zeta)\frac{\Phi_j(z,\zeta;t)}{p_j(\zeta)}\right]-\mbox{div}_\zeta\left[\frac{\Phi_j(z,\zeta;t)}{p_j(\zeta)}\nabla_\zeta a_j(\zeta)\right]$$
is the adjoint operator of $L$ and
$\mbox{div}_z$ and $\nabla_z$ denote respectively the divergence and  the gradient operators w.r. to the variable $z$.

\noindent The mean integrated squared error (MISE) criterion to determine an optimal bandwidth $t$ is given, dropping the index $n$, by
$$MISE_{\bh^{(j)}}(t)=\E_{h^{(j)}}\int_{\C}[\bh^{(j)}(z,t)-h^{(j)}(z,t)]^2dz$$
or $$MISE_{\bh^{(j)}}(t)=\int_{\C}[\E_{h^{(j)}}\{\bh^{(j)}(z,t)\}-h^{(j)}(z)]^2dz+
\int_{\C}Var_h[\bh^{(j)}(z,t)]dz.$$
Following Botev et al. \cite{botev} we have
\begin{theorem}
$$\int_{\C}[\E_{h^{(j)}}\{\bh^{(j)}(z,t)\}-h^{(j)}(z)]^2dz\approx t^2\| L[h^{(j)}]\|^2$$
$$\int_{\C}Var_h[\bh^{(j)}(z,t)]dz\approx \frac{1}{2Rt}\E_{h^{(j)}}[G_j].$$

\end{theorem}

\noindent The MISE is then
$$MISE_{\bh^{(j)}}(t)\approx t^2\| L[h^{(j)}]\|^2+\frac{1}{2Rt}\E_{h^{(j)}}[G_j]$$
which has a unique real positive minimum in
\bb  t_j=  \sqrt[3]{  \frac{\E_{h^{(j)}}[G_j]}{4 R\| L[h^{(j)}]\|^2 }}.\label{ob}\be

\noindent The optimal kernel  estimator of $h^{(j)}_n(z,t)$ is then given by
$$\bh^{(j)}_n(z,t_j)=\frac{1}{R}\sum_{r=1}^R\Phi_j(z,\bzet_j(r);t_j)$$
and the optimal kernel estimator of $h_n(z,t)$ is  given by
\bb\bh_n(z,t_1,\dots,t_{n/2})=\frac{2}{n}\sum_{j=1}^{n/2}\bh^{(j)}_n(z,t_j)\label{oke}\be

\section{The algorithm}

In the following we assume that the generalized eigenvalues $\zeta_j^{(r)},\;j=1\dots n/2$  of the pencils  $[U_1^{(r)},U_0^{(r)}]$ - where the Hankel matrices $U_0^{(r)},U_1^{(r)}$ are based on $d_k^{(r)}$ - have been computed for each  $r=1,\dots,R$ and clustered in such a way that for each $j=1\dots n/2$ the $j-th$ cluster is the set $\{\zeta_j^{(r)},\;r=1,\dots,R\}$ whose elements are independent realization of the r.v. $\bzet_j$. The k-means method \cite{kmean} can be used to solve the clustering problem.

\noindent In order to use the optimal kernel estimator given in eqs. \ref{ob} and \ref{oke}, for each $j=1\dots n/2$ and $r=1,\dots,R$  we need  to solve the initial value problem obtained by eq. \ref{diff} in the limit for $\sigma\downarrow 0$ with initial condition given in eq. \ref{ic}.
By using the transformation $\tilde{h}(x,y,t)=\frac{h(x,y,t)}{p(x,y,t)}$ the equation $h_t=L[h]$ can be rewritten as
$\tilde{h}_t=\frac{1}{p}\mbox{div}\left[{a\nabla{\left(\tilde{h}\right)}}\right]$. Therefore the initial value problems are
\bb\left\{\begin{array}{ll}\frac{\partial\Phi_{jr}}{\partial t}=\frac{1}{p}\left(\Delta \Phi_{jr} + a_x\frac{\partial \Phi_{jr}}{\partial x} + a_y\frac{\partial\Phi_{jr}}{\partial y}\right) \\
       \Phi_{jr}(x,y,0)=\frac{E_j(x,y)}{p(x,y)}\end{array}\right. \label{adiff}\be
where $\Delta$ is the Laplacian operator, $E_j(x,y)$ is the empirical distribution of the generalized eigenvalue $\bzet_j$ and
$$a(x,y;j,r)=\frac{p(x,y) \nu_{jr} (x,y)}{\left(\frac{\partial\nu_{jr}}{\partial x}(x,y)\right)^2+\left(\frac{\partial\nu_{jr}}{\partial y} (x,y)\right)^2},\;\;\nu_{jr}(x,y)=\hat{\us}^H(I_n-X_{jr}^H(X_{jr}X_{jr}^H)^{-1}
X_j)\hat{\us} $$ where
$$\hat{\us}=\frac{1}{R}\sum_{r=1}^R\ud^{(r)},\;\;X_{jr}(h,k)=\overline{x}_{hk}^{(jr)}),\;\;
x_{hk}^{(jr)}=\left\{\begin{array}{ll}
\left(\zeta_h^{(r)}\right)^{k-1}, & h\ne j\\
z^{k-1}, & h=j
\end{array}\right. .$$
The density $p(x,y)=h_n^{(j)}(x,y,\infty)$ is circularly symmetric (Cor. \ref{cor4}). Moreover a closed form model of its modulus for each $n$ can be found in \cite{barja}.

The initial value problems were solved by a collocation method described in \cite{were} in a $m_x\times m_y$ non-uniform grid. For each $(j,r)$ the solution was approximated by the tensor product of Chebyshev polynomials in each spatial variable.  Fast Fourier transform was used to compute the spatial derivatives. The resulting non-linear ODE system
$$ \uw'(t)=F(t,\uw(t)),\;\;\uw(0)=\ue,\;\;\uw(t),\ue\in \R^{m_x m_y}$$
was then solved by MATLAB's built-in function  $\mit ode45.m$,
where  $F$ is the discretized right-hand side of eq. \ref{adiff} and $\ue$ is the discretization of $E_j(x,y)$. The method is fast and stable provided that some spurious oscillations of the spatial derivatives close to the border of the integration region are filtered out. This task is accomplished by multiplying the derivatives by the function
$$
F(x,y)=[ \arctan\{(\tilde{x}+ \gamma)/\phi\}-\arctan\{(\tilde{x} - \gamma)/\phi\} ] [ \arctan\{(\tilde{y} + \gamma)/\phi\}-\arctan\{(\tilde{y} -\gamma)/\phi\} ] $$
where
$$\tilde{x}=\pi (x - x_{min})/(x_{max} - x_{min}) - \pi/2,\;\;
\tilde{y}=\pi (y - y_{min})/(y_{max} - y_{min}) - \pi/2$$ and $\gamma$ and $\phi$ are suitable positive parameters.

\noindent To compute the optimal bandwidth,  $\E_{h^{(j)}}[G_j]$ is estimated by the sample mean of $G_j(\zeta,z)$ i.e. if the computed generalized eigenvalues are denoted by $\zeta_j^{(r)}$ then
$$\E_{h^{(j)}}[G_j]\approx\frac{1}{R}\sum_{r=1}^R\int_{\R^2} G_j(\zeta_j^{(r)},x,y)\hat{h}_n^{(j)}\left(\zeta_j^{(r)},x,y\right)dx dy\approx$$
$$\frac{t_j}{R}\sum_{r=1}^R\sum_{h=1}^{m_x}\sum_{k=1}^{m_y} \left[\Phi_{jr}(x_h,y_k,t_j)\right]^2e^{\frac{\nu_{jr}\left(\zeta_j^{(r)},x_h,y_k\right)}{t_j}}\delta_x(h)\delta_y(k)$$
where  $\delta_x(h)=x_h-x_{h-1},\;\;\delta_x(1)=\delta_x(2),$ and $\delta_y(k)=y_k-y_{k-1},\;\;\delta_y(1)=\delta_y(2)$.
Moreover $$\| L[h_n^{(j)}]\|^2 \approx\left\| \frac{\partial\Phi_{jr}}{\partial t}\right\|^2 \approx\frac{1}{R}\sum_{r=1}^R \sum_{h=1}^{m_x}\sum_{k=1}^{m_y}\left[\frac{\partial\Phi_{jr}(x_h,y_k,\zeta_j^{(r)},t_j)}{\partial t}\right]^2\delta_x(h)\delta_y(k)$$
   To compute $\E_{h^{(j)}}[G_j]$ and $\| L[h^{(j)}]\|^2$  we first need an estimate of $t_j$ which can be provided e.g.  by the variance $\hat{t}_j$ of the generalized eigenvalues in each cluster.

\section{Numerical results}

In order to appreciate the advantages of the proposed kernel estimator, two  numerical experiments were performed.
$R=10$ independent
  realizations $d_k^{(r)}=s_k+\epsilon_k^{(r)},\;\;k=1,\dots,n,\;\;r=1,\dots,R$ of the r.v.  $\bd_k$ were
  generated from the complex exponentials model with $\p=5$ components given by
$$\underline{\xi}=\left[ e^{-0.1-i 2\pi  0.3},e^{-0.05-i 2\pi
0.28},e^{-0.0001+i 2\pi 0.2},e^{-0.0001+i 2\pi  0.21},e^{-0.3-i 2\pi
0.35}\right]$$ $$ \underline{c}=\left[ 6,3,1,1,20\right
],\;\;n=74,\;p=37,\;\sigma=1,\;\sigma=3.$$   We notice that  the frequencies of the $4^{rd}$
and $5^{th}$ components are closer than the Nyquist frequency
if $n<1/(0.21-0.20)=100$. Therefore a super-resolution problem has to be solved. To speed up the computations  we limit the analysis of the condensed density to two regions containing respectively the first and second components and the third and fourth ones. The fifth component is isolated with a large amplitude therefore it is easy to identify even if its decay is fast. The considered regions are the rectangles defined by   $\Omega=\Omega_x\times\Omega_y$ where $\Omega_x=[-0.8,0.4],\;\;\Omega_y=[-1.4,-0.4]$ for the first region  and $\Omega_x=[-0.1,0.6],\;\;\Omega_y=[0.5,1.3]$ for the second one. A mesh of size $m_x=64,\;m_y=64$ was considered. The values  $\gamma=1.6,\;\;\phi=0.02 $ of the derivatives filter were selected by trials and errors.

In order to apply the proposed method a  pilot density estimate was first computed  by  the closed form approximation method given in \cite{distrf}
$$\hat{h}(z)\propto
\sum_{r=1}^N\sum_{k=1}^p\hat{\Delta}
\left\{\Psi\left[\left(\frac{R_{kk}^{(r)}(z)^2}{\sigma^2\beta}+1\right)\right]\right\}
$$ where $R^{(r)}(z)$
is the $R-$ factor of  the QR factorization of the matrix $U_1^{(r)}-zU_0^{(r)}$ and $\beta=5 n\sigma^2$ \cite[Prop.6]{dsp10}.
Then the  generalized eigenvalues of the pencils $[U_1^{(r)},U_0^{(r)}],\;r=1,\dots,R$ were pooled and the k-means method of clustering was  applied with the number of clusters equal to the number of relative maxima of the pilot estimate.

For comparison a Gaussian kernel estimate of density was also computed  by Algorithm 1 in \cite[App.E]{botev}.

In figure \ref{fig1} the results obtained for the first region when $\sigma=1$ are plotted:  the empirical density (top left),  the pilot density computed by the closed form approximation method (top right),  the Gaussian kernel  estimate (bottom left) and the result obtained by the proposed method (bottom right). In figure \ref{fig2} the same results are plotted for the second region. In figures \ref{fig3} and \ref{fig4} the results obtained when $\sigma=3$ are reported.
The positions of the true complex exponentials are marked by a cross. It can be noticed that the proposed method is able to identify the two true complex exponentials even in the worst case, filtering out most of the spurious peaks of the empirical condensed density. Even if the location of the peaks is not perfect, it is the only method which is able to provide a reasonable solution to the super-resolution problem  in the second region also for the smallest SNR considered.


\begin{figure}[H]
\hspace*{-0.5in}
\includegraphics[totalheight=6.in]{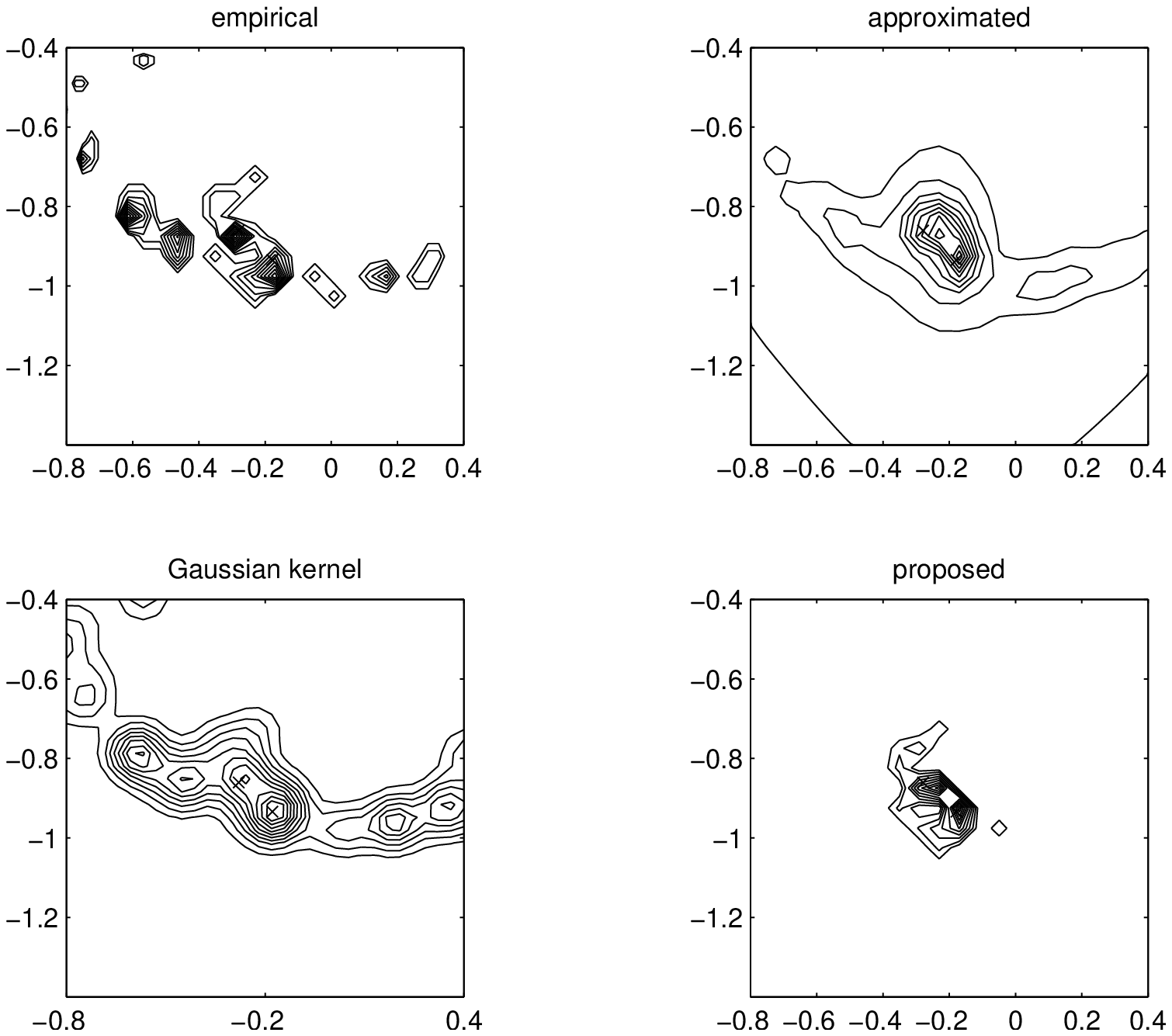}
\caption{First region, $\sigma=1$. The empirical density (top left),  the pilot density computed by the closed form approximation method (top right),  the Gaussian kernel estimate (bottom left) and the result obtained by the proposed method (bottom right)}
\label{fig1}
\end{figure}

\begin{figure}[H]
\hspace*{-0.5in}
\includegraphics[totalheight=6.in]{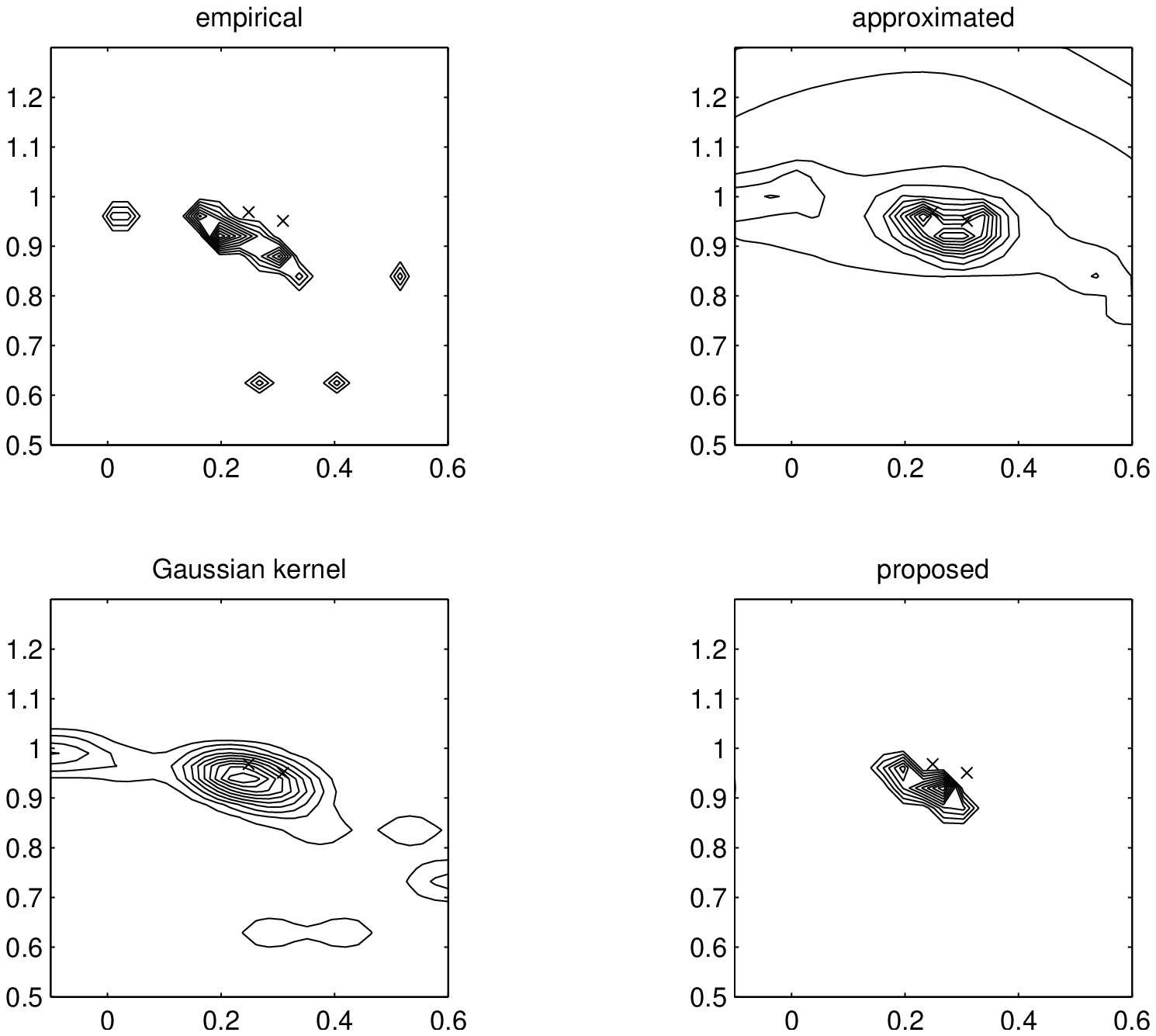}
\caption{Second region, $\sigma=1$. The empirical density (top left),  the pilot density computed by the closed form approximation method (top right),  the Gaussian kernel estimate (bottom left) and the result obtained by the proposed method (bottom right)}
\label{fig2}
\end{figure}

\begin{figure}[H]
\hspace*{-0.5in}
\includegraphics[totalheight=6.in]{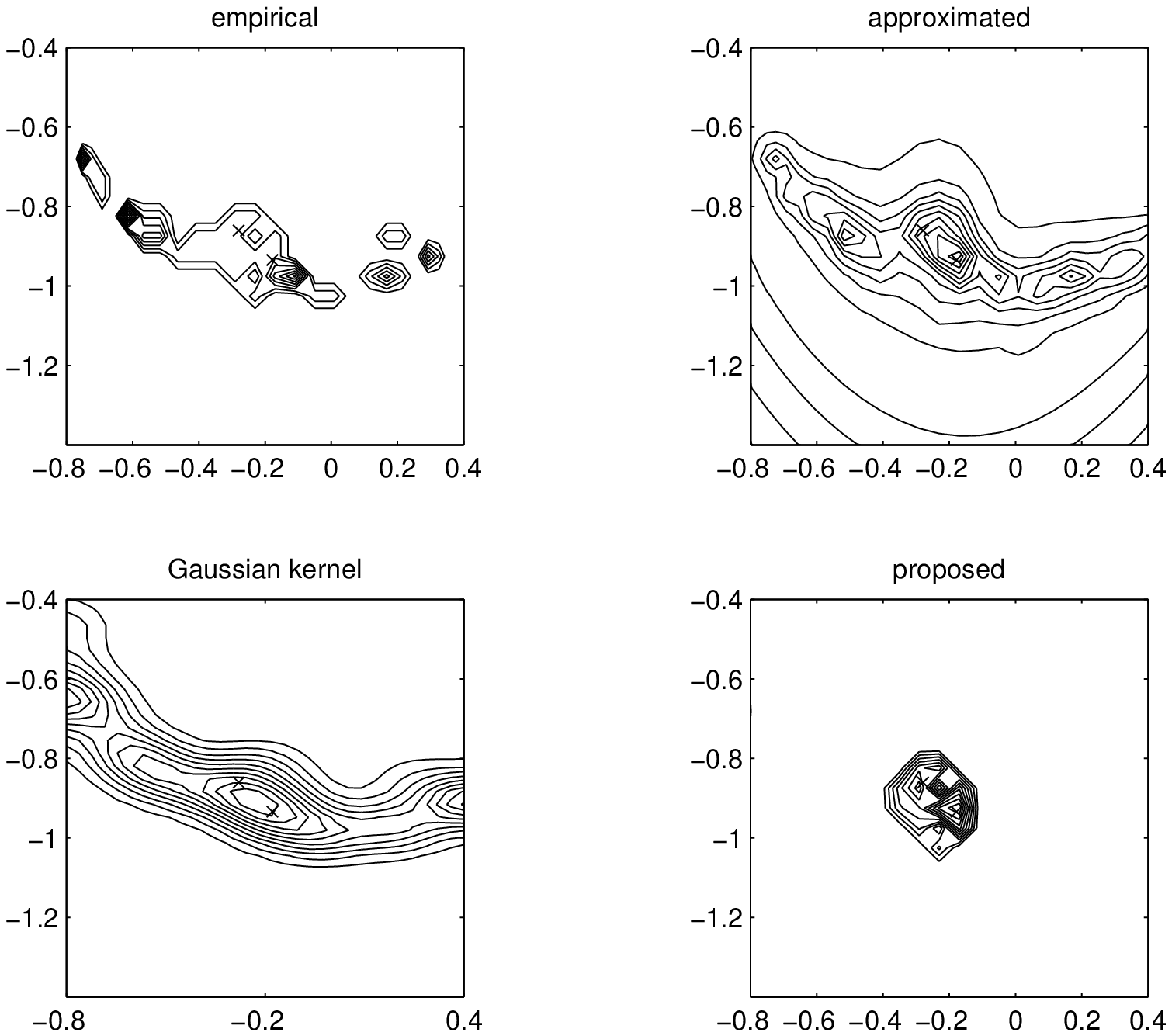}
\caption{First region, $\sigma=3$. The empirical density (top left),  the pilot density computed by the closed form approximation method (top right),  the Gaussian kernel estimate (bottom left) and the result obtained by the proposed method (bottom right)}
\label{fig3}
\end{figure}

\begin{figure}[H]
\hspace*{-0.5in}
\includegraphics[totalheight=6.in]{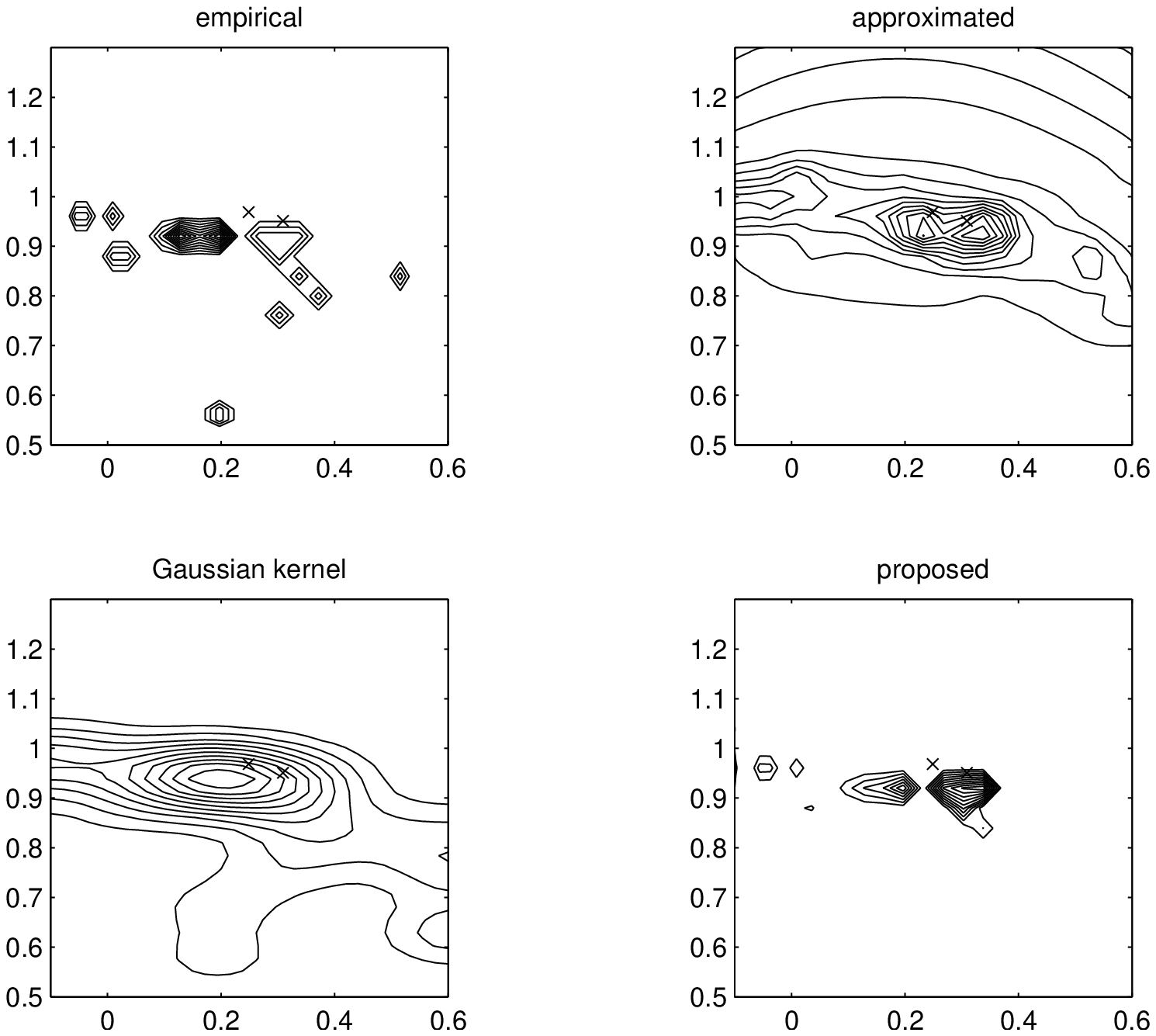}
\caption{Second region, $\sigma=3$. The empirical density (top left),  the pilot density computed by the closed form approximation method (top right),  the Gaussian kernel estimate (bottom left) and the result obtained by the proposed method (bottom right)}
\label{fig4}
\end{figure}

\section*{Appendix}

\subsection*{Proof of Lemma 1.1}

 Let us find $\umu_j \in \C^{n/2}$, $Q_j \in \C^{n/2 \times n/2}$ and $\nu_j\in\C$ such that
$$
\sum_{k=0}^{n-1}\left|\sum_{h\ne j}^{1,n/2} \gamma_h\zeta_h^{k}+\gamma_j
z^{k}-s_k\right|^2=(\ugam-\umu_j)^HQ_j(\ugam-\umu_j)+\nu_j$$
We have
$$
\sum_{k=0}^{n-1}\left|\sum_{h\ne j}^{1,n/2} \gamma_h\zeta_h^{k}+\gamma_j
z^{k}-s_k\right|^2=\sum_{k=0}^{n-1}\left|\sum_{h}^{1,n/2}\gamma_h x_{hk}^{(j)}-s_k\right|^2.$$
Choosing  $\umu_j, Q_j$ an $\nu_j$ as defined in the text of the Lemma
we get the result.

\subsection*{Proof of Lemma 1.3}

  From Lemma \ref{lem1} we have
  $$h_2^{(1)}(z,\sigma)= \frac{1}{(\pi \sigma^2)^{2}}\int_{\C}
\gamma e^{-\frac{1}{\sigma^2}[(\gamma-\mu_1)^HQ_1(\gamma-\mu_1)+\nu_1]
} d\gamma$$
  where $X_1=[1\;\; \overline{z}],\;\;Q_1=1+|z^2|,\;\;\mu_1=\frac{s_0+\overline{z}s_1}{1+|z^2|},\;\;\nu_1=|s_0|^2+|s_1|^2-\frac{|s_0+\overline{z}s_1|^2}{1+|z^2|}$,
  therefore
 $$h_2^{(1)}(z,\sigma)= \frac{1}{(\pi \sigma^2)^{2}}e^{-\frac{1}{\sigma^2}\nu_1}\int_{\C}
\gamma e^{-\frac{(\gamma-\mu_1)^H(\gamma-\mu_1)(1+|z^2|)}{\sigma^2}
} d\gamma=$$ $$\frac{1}{(\pi \sigma^2)^{2}}e^{-\frac{1}{\sigma^2}\nu_1}\int_{\R^2}
|\gamma|^2 e^{-\frac{(\gamma-\mu_1)^H(\gamma-\mu_1)(1+|z^2|)}{\sigma^2}
} d\Re{\gamma}d\Im{\gamma}=$$
$$\frac{1}{(\pi \sigma^2)^{2}}e^{-\frac{1}{\sigma^2}\nu_1}\frac{\pi \sigma^2}{1+|z|^2}\left(|\mu_1|^2+\frac{\sigma^2}{1+|z|^2}\right).$$
If $\xi=\frac{s_1}{s_0}$ then
$$\nu_1=|s_0|^2\frac{|z-\xi|^2}{1+|z|^2},\;\;\mbox{and}\;\;|\mu_1|^2=|s_0|^2\frac{|1+\overline{z}\xi|^2}{(1+|z|^2)^2}$$
and
$$h_2^{(1)}(z,\sigma)=e^{-\frac{|s_0|^2}{\sigma^2}\frac{|z-\xi|^2}{1+|z|^2}}\left(\frac{|s_0|^2|1+\overline{z}\xi|^2}{\pi \sigma^2 (1+|z|^2)^3} + \frac{1}{\pi(1+|z|^2)^2}  \right).$$
But $s_0=c_1$ and $s_1=c_1\xi_1$ hence $\xi_1=\xi$ and we get the thesis.

\subsection*{Proof of Lemma 1.4}

 By considering the vector $\tilde{\umu}_j$ obtained by stacking the real and imaginary parts of  and $\umu$ , and the real isomorph
$\tilde{Q}_j$
of the matrix $Q_j$, and remembering that the Jacobian with respect to the real and imaginary part of a complex variable is $J_R=|J_C|^2$, we get from Lemma \ref{lem1} for $n>2$
\bb && h_n^{(j)}(z,\sigma)=\frac{1}{(\pi\sigma^2)^{n}}\int_{\R^{n-2}}\int_{\R^{n}}
|J_C^*|^2e^{-\frac{1}{\sigma^2}[(\tilde{\ugam}-\tilde{\umu}_j)^H\tilde{Q}_j(\tilde{\ugam}-\tilde{\umu}_j)+\nu_j]}d \tilde{\ugam }
d\Re{\uzet^{(j)}}d\Im{\uzet^{(j)}}= \nonumber\\&&
\frac{1}{(\pi\sigma^2)^{n}}\int_{\R^{n-2}}\int_{\R^{n}}
\prod_{ j=1}^{1,n/2}|\gamma_j|^2\prod_{r<h;r,h\ne
j}|\zeta_r-\zeta_h|^8\prod_{r\ne j}|\zeta_r-z|^8e^{-\frac{1}{\sigma^2}[(\tilde{\ugam}-\tilde{\umu}_j)^H\tilde{Q}_j(\tilde{\ugam}-\tilde{\umu}_j)+\nu_j]}d \tilde{\ugam }
d\Re{\uzet^{(j)}}d\Im{\uzet^{(j)}}.\nonumber
\be
By defining
$$f_j(\uzet^{(j)},z,\sigma)=\frac{1}{(\pi\sigma^2)^{n/2}|\tilde{Q}_j^{-1}|^{\frac{1}{2}}}\int_{\R^{n}}\left( \prod_{ i=1}^{1,n/2} \tilde{\ugam}^T A_i \tilde{\ugam}\right)
e^{-\frac{1}{\sigma^2}(\tilde{\ugam}-\tilde{\umu}_j)^H\tilde{Q}_j(\tilde{\ugam}-\tilde{\umu}_j)} d\tilde{\ugam}  $$
and noticing that
$$|\gamma_i |^2= \tilde{\ugam}^T A_i\tilde{\ugam}$$
we have
$$h_n^{(j)}(z,\sigma)=\int_{\R^{n-2}}f_j(\uzet^{(j)},z)\frac{\prod_{r<h;r,h\ne j}|\zeta_r-\zeta_h|^8\prod_{r\ne j}|\zeta_r-z|^8 }{(\pi\sigma^2)^{n/2}|\tilde{Q}_j(\uzet^{(j)},z)|^{\frac{1}{2}}}e^{-\frac{1}{\sigma^2}\nu_j}
d\Re{\uzet^{(j)}}d\Im{\uzet^{(j)}}$$
and the thesis follows.

\subsection*{Proof of Lemma 1.5}

By considering the normalized vectors $\uy_j= \Sigma_j^{-\frac{1}{2}}\tilde{\ugam}$, we have
$\uy_j \sim N\left(\um_j,I\right)$ where $$\um_j=\Sigma_j^{-\frac{1}{2}}\tilde{\umu}_j=\frac{\sqrt{2}}{\sigma}\tilde{Q}_j^{1/2}\tilde{Q}_j^{-1}\tilde{\uq}_j=
\frac{\sqrt{2}}{\sigma}\tilde{Q}_j^{-1/2}\tilde{\uq}_j=\frac{\sqrt{2}}{\sigma D_j^{1/2}}\hat{Q}_j^{1/2}\tilde{\uq}_j$$ where $D_j=\det(\tilde{Q}_j)$ and $\hat{Q}_j=adj(\tilde{Q}_j)$.
Moreover
$$\tilde{\ugam}^H A_i \tilde{\ugam}=(\Sigma_j^{-\frac{1}{2}}\tilde{\ugam})^T \Sigma_j^{\frac{1}{2}} A_i \Sigma_j^{\frac{1}{2}} \Sigma_j^{-\frac{1}{2}}\tilde{\ugam}=\uy_j^T \Sigma_j^{\frac{1}{2}} A_i \Sigma_j^{\frac{1}{2}} \uy_j=\uy_j^TB_i \uy_j$$
where $$B_i=\Sigma_j^{\frac{1}{2}} A_i \Sigma_j^{\frac{1}{2 }}=\frac{\sigma^2}{2}\tilde{Q}_j^{-1/2}A_i\tilde{Q}_j^{-1/2}=\frac{\sigma^2}{2D_j}\hat{Q}_j^{1/2}A_i\hat{Q}_j^{1/2}.$$
From \cite[Th.1]{bu}, denoting by ${\cal Q}_i$ the quadratic form $\uy_j^TB_i\uy_j$, we have the recursion
$$f_j(\uzet^{(j)},z,\sigma)=
E\left[\prod_{ i=1}^{1,n/2}{\cal Q}_i\right]=\sum_{i=0}^{n/2-1}2^i\sum_{j_1=2}^{n/2}\dots\sum_{j_i=2}^{n/2}
\left(g_{j_1\dots j_i}E\left[\frac{{\cal Q}_2\dots{\cal Q}_{n/2}}{{\cal Q}_{j_1}\dots{\cal Q}_{j_i}} \right]\right)$$
where  for $i=0$, $g=\um^TB_1\um+tr(B_1)=E[{\cal Q}_1]$ and for $i>0$, $j_1\ne j_2\ne\dots\ne j_i$ and
$$g_{j_1\dots j_i}= \um^T(B_1B_{j_1}\dots B_{j_i}+B_{j_1}B_1B_{j_2}\dots B_{j_i}+\dots + B_{j_1}B_{j_2}\dots B_{j_i}B_1)\um+tr(B_1B_{j_1}\dots B_{j_i}). $$
But then
$$g_{j_1\dots j_i}=\frac{\sigma^{2(i-1)}}{2^{i-1}D_j^{i+1}}\tilde{\uq}_j^T{\cal A}\tilde{\uq}_j+
\frac{\sigma^{2i}}{2^iD_j^i}tr(A_1\hat{Q}_j
A_{j_1}\hat{Q}_j\dots A_{j_i}\hat{Q}_j)$$
$$g_{j_1\dots j_i}=\frac{\sigma^{2i}}{2^iD_j^i}\left(\frac{2}{\sigma^2 D_j}\hat{Q}_j^{1/2}\tilde{\uq}_j^T{\cal A}\tilde{\uq}_j+
tr(A_1\hat{Q}_j
A_{j_1}\hat{Q}_j\dots A_{j_i}\hat{Q}_j)\right)$$
where $${\cal A}=\hat{Q}_jA_1\hat{Q}_j
A_{j_1}\hat{Q}_j\dots \hat{Q}_jA_{j_i}\hat{Q}_j+
\hat{Q}_jA_{j_1}\hat{Q}_jA_1\hat{Q}_jA_{j_2}\hat{Q}_j\dots A_{j_i}\hat{Q}_j+\dots +
\hat{Q}_jA_{j_1}\hat{Q}_j\dots A_{j_i}\hat{Q}_jA_1\hat{Q}_j.$$
We have
$$E[{\cal Q}_1]= \frac{1}{D_j^2}\tilde{\uq}_j^T\hat{Q}_jA_1\hat{Q}_j\tilde{\uq}_j+\frac{\sigma^2}{2D_j}tr(A_1\hat{Q}_j)$$
$$E[{\cal Q}_1{\cal Q}_2]= E[{\cal Q}_1]E[{\cal Q}_2] +4 \um_j^TB_1B_2\um_j+2tr(B_1B_2)=$$
$$\frac{1}{D_j^4}\tilde{\uq}_j^T\hat{Q}_jA_1\hat{Q}_j\tilde{\uq}_j\tilde{\uq}_j^T\hat{Q}_jA_2\hat{Q}_j\tilde{\uq}_j+$$
$$\frac{\sigma^2}{2D_j^3}\left(\tilde{\uq}_j^T\hat{Q}_jA_1\hat{Q}_j\tilde{\uq}_jtr(A_2\hat{Q}_j)+\tilde{\uq}_j^T\hat{Q}_jA_2\hat{Q}_j\tilde{\uq}_jtr(A_1\hat{Q}_j)+\tilde{\uq}_j^T\hat{Q}_jA_i\hat{Q}_jA_i\hat{Q}_j\tilde{\uq}_j\right)+$$
$$\frac{\sigma^4}{4D_j^2}\left(tr(A_1\hat{Q}_j)tr(A_2\hat{Q}_j)+tr(A_1\hat{Q}_jA_2\hat{Q}_j)\right)$$
and in general
$$f_j(\uzet^{(j)},z,\sigma)=P_{n/2}(\sigma^{2},z)=\sum_{h=0}^{n/2}\beta_h(\uzet^{(j)},z)\frac{\sigma^{2h}}{2^h D_j^{n-h}}$$
where $\beta_h(\uzet^{(j)},z)$ are positive polynomials. Moreover $\beta_{n/2}(\uzet^{(j)},z)$ is the only coefficient that does not depend on $\tilde{\uq}_j$ and therefore it  does not depend on $\us$.

\subsection*{Proof of Corollary 1.6}

By Lemma \ref{lem2}
$$h_n^{(j)}(z,\sigma)=\int_{\R^{n-2}}f_j(\uzet^{(j)},z)\frac{\prod_{r<h;r,h\ne j}|\zeta_r-\zeta_h|^8\prod_{r\ne j}|\zeta_r-z|^8 }{(\pi\sigma^2)^{n/2}|\tilde{Q}_j(\uzet^{(j)},z)|^{\frac{1}{2}}}e^{-\frac{1}{\sigma^2}\nu_j}
d\Re{\uzet^{(j)}}d\Im{\uzet^{(j)}}=$$
$$\frac{1}{\sigma^n\pi^{n/2}}\sum_{k=0}^{n/2}\frac{\sigma^{2k}}{2^k}\int_{\R^{n-2}}\beta_k(\uzet^{(j)},z)\frac{\prod_{r<h;r,h\ne j}|\zeta_r-\zeta_h|^8\prod_{r\ne j}|\zeta_r-z|^8 }{D_j^{n-k+\frac{1}{2}}}e^{-\frac{1}{\sigma^2}\nu_j}
d\Re{\uzet^{(j)}}d\Im{\uzet^{(j)}}.$$
Noting that $|\zeta_r-\zeta_h|^8=|\tzet_r-\tzet_h|^8$ and $|\zeta_r-z|^8=|\tzet_r-\tz|^8$, we get  the thesis.

\subsection*{Proof of Corollary 1.7}

Noticing that in the definition of $g_j(\tilde{\uzet}^{(j)},\tz,\sigma)$ when $\sigma\downarrow\infty$ all terms vanish but the last one, we get the first part of the thesis.
By Lemma \ref{lemf} $\beta_{n/2}(\tilde{\uzet}^{(j)},\tz)$ is the only coefficient which does not depend on  $\us$. Therefore when $\us=\underline{0}$ equation \ref{eqdc} reduces to  $ h_n^{(j)}(\tz,\infty)$. Finally by symmetry the condensed density does not depend on $j$, therefore all $h_n^{(j)}(\tz,\infty)$ must be equal. Moreover in \cite{barja} it was proved that, when $\us=\underline{0}$, it depends only on $|\tz|^2$.

\subsection*{Proof of Theorem 2.1}

We recall that if $$I=\int_{\uy\in {\cal D}}K(\uy)e^{-\lambda\nu(\uy)}d\uy,\;\;{\cal D}\mbox{ open set }\subset\R^d,\;\;\lambda\in \R^+$$
and $\nu(\uy)$ has a unique minimum in $\overline{{\cal D}}$ and this minimum occurs at a stationary point $\hat{\uy}$ of $\nu(\uy)$, then the Laplace's approximation to $I$ is given by
$$\tilde{I}=(2 \pi)^{\frac{d}{2}}\lambda^{-\frac{d}{2}}|H(\hat{\uy})|^{-\frac{1}{2}}K(\hat{\uy})e^{-\lambda\nu(\hat{\uy})}$$
where $H(\uy)$ is the Hessian of $\nu$.

\noindent We know that $\nu_j\ge 0$ and, by Lemma \ref{cor1}, $\uxi$ is the only vector such that $\nu_j(\uxi)=0$. Therefore by continuity, $\nu_j(\tilde{\uzet}^{(j)},\xi_j)$  has a unique minimum $\hat{\uzet}_j$  in a neighbor ${\cal D}_j$ of $\{\xi_h,\;h\ne j\}$.
Moreover from  Corollary \ref{lem3} we notice that the dominant term in the sum defining $g_j(\tilde{\uzet}^{(j)},\tz,\sigma)$ when $\sigma\downarrow 0$ is the first one, therefore
in this case $$g_j(\tilde{\uzet}^{(j)},\tz,\sigma)\approx\frac{1}{\sigma^{n}\pi^{n/2}}\beta_0(\tilde{\uzet}^{(j)},\tz)\frac{\prod_{r<h;r,h\ne j}|\tzet_r-\tzet_h|^8\prod_{r\ne j}|\tzet_r-\tz|^8 }{D_j^{n+\frac{1}{2}}(\tilde{\uzet}^{(j)},\tz)}=\frac{1}{\sigma^n}K_j(\tilde{\uzet}^{(j)},\tz).$$
Then by using Laplace's approximation with $\lambda=\frac{1}{\sigma^2}$ and $d=n-2$, we have, for $\tz\in {\cal N}_j$ where ${\cal N}_j$ is a neighbor of $\xi_j$
\bb \hat{h}_n^{(j)}(\tz,\sigma)= \frac{1}{\sigma^n}\int_{D_j} K_j(\tilde{\uzet}^{(j)},\tz) e^{-\frac{1}{\sigma^2}\nu_j(\tilde{\uzet}^{(j)},\tz)} d\tilde{\uzet}^{(j)}\nonumber \\ \approx \sigma^{-2}(2 \pi)^{\frac{n}{2}-1} |H_j(\hat{\uzet}_j,z)|^{-\frac{1}{2}}K_j(\hat{\uzet}_j,\tz)e^{-\frac{1}{\sigma^2}\nu_j(\hat{\uzet_j},\tz)}
\be
where $H_j$ is the Hessian of $\nu_j$. For simplicity we will denote this approximation by the same symbol $\hat{h}_n^{(j)}(z,\sigma)$.
Let us define $$G_j(\hat{\uzet}^{(j)},\tz)=(2 \pi)^{\frac{n}{2}-1} |H_j(\hat{\uzet}_j,\tz)|^{-\frac{1}{2}}K_j(\hat{\uzet}_j,\tz)$$ then
\bb \hat{h}_n^{(j)}(\tz,\sigma)=\sigma^{-2}G_j(\hat{\uzet}^{(j)},\tz)e^{-\frac{1}{\sigma^2}\nu_j(\hat{\uzet}^{(j)},\tz)}.\be

\subsection*{Proof of Lemma 3.1}

   $$\frac{\partial {\cal D} (h,p)}{\partial t}=\int_{\R^2} p(x,y)\frac{\partial}{\partial t}\left[\Psi\left(\frac{h(x,y,t)}{p(x,y)}\right)\right]dxdy=$$
   $$ \int_{\R^2}\Psi'\left(\frac{h(x,y,t)}{p(x,y)}\right) h_t(x,y)dxdy=
   \int_{\R^2}\Psi'\left(\frac{h(x,y,t)}{p(x,y)}\right)\mbox{div}\left[a(x,y)\nabla\left(\frac{h(x,y,t)}{p(x,y)}\right)\right] dxdy.$$
   Integrating by parts we get
   $$\frac{\partial {\cal D} (h,p)}{\partial t}=\int_{\R^2}\mbox{div}\left[\Psi'\left(\frac{h(x,y,t)}{p(x,y)}\right)a(x,y)\nabla\left(\frac{h(x,y,t)}{p(x,y)}\right)\right] dxdy - $$ $$ \int_{\R^2}\nabla\left[\Psi'\left(\frac{h(x,y,t)}{p(x,y)}\right)\right]\cdot a(x,y)\nabla{\left(\frac{h(x,y,t)}{p(x,y)}\right)}dxdy,$$ where $\cdot$ denotes scalar product.
   By the divergence theorem the first integral is zero because $\Psi'(1)=0$. Moreover
   $$\nabla\left[\Psi'\left(\frac{h(x,y,t)}{p(x,y)}\right)\right]= \psi
   ''\left(\frac{h(x,y,t)}{p(x,y)}\right)\nabla{\left(\frac{h(x,y,t)}{p(x,y)}\right)}.$$

\subsection*{Proof of Theorem 4.1}

$$\frac{\partial\E_{h^{(j)}}\{\bh^{(j)}(z,t)\}}{\partial t}=\int_{\C}\frac{\partial \Phi_j}{\partial t}(z,\zeta;t)h^{(j)}(\zeta)d\zeta=\int_{\C}L^*[\Phi_j(z,\zeta;t)]h^{(j)}(\zeta)d\zeta.$$
By  definition of  adjoint operator, taking into account that  $\lim_{z\rightarrow \infty}\Phi_j(z,\zeta;t)=0$ we get for $t\downarrow 0$
$$\frac{\partial\E_{h^{(j)}}\{\bh^{(j)}(z,t)\}}{\partial t}=\int_{\C}L^*[\Phi_j(z,\zeta;t)]h^{(j)}(\zeta)d\zeta=\int_{\C}\Phi_j(z,\zeta;t)L[h^{(j)}(\zeta)]d\zeta.$$
But $\lim_{t\rightarrow 0}\Phi_j(z,\zeta;t)=\delta(z-\zeta)$ hence
$$\int_{\C}\Phi_j(z,\zeta;t)L[h^{(j)}(\zeta)]d\zeta\approx L[h^{(j)}(z)],\;\;t\downarrow 0.$$
By considering the first order Taylor series approximation of $\E_{h^{(j)}}\{\bh^{(j)}(z,t)\}$ we get
$$\E_{h^{(j)}}\{\bh^{(j)}(z,t)\}=\E_{h^{(j)}}\{\bh^{(j)}(z,0)\}+t \frac{\partial\E_{h^{(j)}}\{\bh^{(j)}(z,t)\}}{\partial t}|_{t=0}+O(t^2),$$
but
$$\E_{h^{(j)}}\{\bh^{(j)}(z,0)\}=\int_{\C}\Phi_j(z,\zeta;0)h^{(j)}(\zeta)d\zeta=\int_{\C}\delta(z-\zeta)h^{(j)}(\zeta)d\zeta=h^{(j)}(z)$$
hence
$$\E_{h^{(j)}}\{\bh^{(j)}(z,t)\}=h^{(j)}(z)+t L[h^{(j)}(z)]+O(t^2)$$
and
$$\int_{\C}[\E_{h^{(j)}}\{\bh^{(j)}(z,t)\}-h^{(j)}(z)]^2dz\approx\int_{\C}[t L[h^{(j)}(z)]]^2dz=t^2\| L[h^{(j)}]\|^2$$

\noindent For approximating the integrated variance let us consider first the second moment
$$\E_{h^{(j)}}\{\Phi_j(z,\zeta;t)^2\}=\int_{\C}\Phi_j(z,\zeta;t)^2h^{(j)}(\zeta)d\zeta=
 \int_{\C}t^{-2}G_j(z,\zeta)^2e^{-\frac{2}{t}\nu_j(z,\zeta)}h^{(j)}(\zeta)d\zeta$$
But it was proved in \cite[Th.2]{barja2} that
$$\lim_{t\rightarrow 0}\frac{2}{t}G_j(z,\zeta)e^{-\frac{2}{t}\nu_j(z,\zeta)}=\delta(z-\zeta)$$ hence, for $t\downarrow 0$,
$$\E_{h^{(j)}}\{\Phi_j(z,\zeta;t)^2\}\approx\frac{1}{2}\int_{\C}t^{-1}G_j(z,\zeta) \delta(z-\zeta)h^{(j)}(\zeta)d\zeta=\frac{1}{2}t^{-1}G_j(z,z)h^{(j)}(z).$$
 As $\bzet_j(r)$ are independent $\forall r.$, it follows that
$$Var_h[\bh^{(j)}(z,t)]=Var_h\left[\frac{1}{R}\sum_{r=1}^R\Phi_j(x,y,\bzet_j(r);t)\right] =\frac{1}{R^2}\sum_{r=1}^R Var_h\left[\Phi_j(x,y,\bzet_j(r);t)\right] =$$
$$\frac{1}{R}\E_{h^{(j)}}[\Phi_j(x,y,\bzet_j(r);t)^2]-\frac{1}{R}(\E_{h^{(j)}}[\Phi_j(x,y,\bzet_j(r);t)])^2\approx $$ $$\frac{1}{2Rt}G_j(z,z)h^{(j)}(z)-(h^{(j)}(z)+t L[h^{(j)}(z)])^2
\approx\frac{1}{2Rt}G_j(z,z)h^{(j)}(z)$$
because for $t\downarrow 0$ the second term is negligible w.r. to to the first one and
$$\int_{\C}Var_h[\bh^{(j)}(z,t)]dz\approx\frac{1}{2Rt}\int_{\C}G_j(z,z)h^{(j)}(z)dz=\frac{1}{2Rt}\E_{h^{(j)}}[G_j].$$

\end{document}